\documentclass[11pt]{article}
\usepackage{amssymb, amsmath, amsthm, amsfonts, enumerate}
\usepackage{amsmath,setspace,scalefnt}
\usepackage[usenames,dvipsnames,svgnames,table]{xcolor}
\usepackage{graphicx,tikz,caption,subcaption}
\usetikzlibrary{patterns}
\usetikzlibrary{shapes}
\usetikzlibrary{decorations.pathreplacing}
\usetikzlibrary{decorations.pathmorphing}
\usetikzlibrary{positioning}
\usepackage{fullpage}
\usepackage{hyperref}

\usepackage{enumitem,verbatim}
\usepackage{bbm,dsfont}

\definecolor{gray}{rgb}{0.25, 0.25, 0.25}

\newtheorem{theorem}{Theorem}
\newtheorem{lemma}[theorem]{Lemma}

\theoremstyle{definition}

\theoremstyle{plain}

\theoremstyle{definition}

\theoremstyle{definition}

\theoremstyle{definition}

\theoremstyle{definition}
\newtheorem{defn}[theorem]{Definition}
\theoremstyle{definition}

\theoremstyle{definition}

\theoremstyle{definition}
\newtheorem{fact}[theorem]{Fact}

\newcommand*{\eqdef}{\stackrel{\mbox{\normalfont\tiny def}}{=}} 

\newenvironment{@abssec}[1]{%
     \if@twocolumn
       \section*{#1}%
     \else
       \vspace{.05in}\footnotesize
       \parindent .2in
         {\upshape\bfseries #1. }\ignorespaces 
     \fi}
     {\if@twocolumn\else\par\vspace{.1in}\fi}

\newcommand\ex{\ensuremath{\mathrm{ex}}}

\DeclareMathOperator{\codeg}{codeg}

\newcommand{\makenote}[2]
{

\smallskip

\noindent

\fbox{
\begin{minipage}{0.95\textwidth}

\def\temp{#1}
\ifx\temp\empty
\def\forlabel{$\bullet$}
\else
\def\forlabel{{\bfseries #1:}}
\fi

\begin{itemize}[label = \forlabel]
  #2
\end{itemize}
\end{minipage}
}

\smallskip
}

\definecolor{darkslateblue}{rgb}{0.28, 0.24, 0.55}
\definecolor{celestialblue}{rgb}{0.29, 0.59, 0.82}
\definecolor{cadmiumgreen}{rgb}{0.0, 0.42, 0.24}
\definecolor{darkpastelgreen}{rgb}{0.01, 0.75, 0.24}
\definecolor{deepcarmine}{rgb}{0.66, 0.13, 0.24}

\title{Positive co-degree Tur\'an number for $C_5$ and $C_5^{-}$}
\author{
Zhuo Wu\\
	\small Mathematics Institute and DIMAP\\[-0.8ex]
	\small University of Warwick\\[-0.8ex]
	\small Coventry, UK.\\ 
}

\begin{document}
\maketitle

\begin{abstract}
The \emph{minimum positive co-degree} $\delta^{+}_{r-1}(H)$ of a non-empty $r$-graph $H$ is the maximum $k$ such that if $S$ is an $(r-1)$-set contained in a hyperedge of $H$, then $S$ is contained in at least $k$ hyperedges of $H$. For any $r$-graph $F$, the \emph{positive degree Tur\'an number} $\mathrm{co}^{+}\ex(n,F)$ is defined as the maximum value of $\delta^{+}_{r-1}(H)$ over all $n$-vertex $F$-free non-empty $r$-graphs $H$. In this paper, we determine the positive degree Tur\'an number for $C_5$ and $C_5^{-}$.
\end{abstract}

\begin{section}{Introduction} 
An $r$-graph is a hypergraph whose edges all have size $r$. Given a $r$-graph $F$, the \emph{Tur\'an number} number $\mathrm{ex}(n,F)$ is the maximum number of edges over all $n$-vertex $F$-free $r$-graphs. The \emph{Tur\'an density} of $F$, denoted by $\pi(F)$, is the limit
\[\pi(F)\eqdef\lim_{n\rightarrow\infty}\frac{\mathrm{ex}(n,F)}{\binom{n}{r}}.\]

Determining these numbers is a central problem in extremal combinatorics. When $r=2$,  Erd\H{o}s-Stone theorem determines the Tur\'an density of  graphs with chromatic number at least 3(\cite{Stone}). However, very few exact results are known for hypergraphs. For example, the Tur\'an density for $K_4^{3}$ (the complete 3-uniform hypergraph on 4 vertices) is unknown.

In order to get a better understanding of these problems, mathematicians introduced new versions of the Tur\'an function. Let $H = (V, E)$ be an $r$-graph and $S$ be a subset of $V$ on $r-1$ vertices. We denote by $\codeg(S)$ the \emph{co-degree} of $S$, which is the number of edges from $E$ that contains $S$. Set 
\[
\delta_{r-1}(H) \eqdef \min_{S \in \binom{V}{r-1}} \codeg(S). 
\]
In 2006, Mubayi and Zhao (\cite{zhao}) considered the \emph{co-degree Tur\'an number} $\mathrm{co}\ex(n,F)$ as the maximum value of $\delta_{r-1}(H)$ over all $n$-vertex $F$-free $r$-graphs $H$, and the \emph{co-degree Tur\'an density}

\[\gamma(F)\eqdef\lim_{n\rightarrow\infty}\frac{\mathrm{coex}(n,F)}{n}.\]

There are quite a few results focusing on deciding the co-degree Tur\'an numbers for specific hypergraphs. For instance, 
\begin{itemize}
    \item Falgas-Ravry, Marchant, Pikhurko and Vaughan proved $\gamma(F_{3,2})=1/3$ where $F_{3,2}$ is the hypergrpah $\{abc,ade,bde,cde\}$ (\cite{Oleg}). 
    \item Mubayi proved that $\gamma(\mathbb{F})=1/2$ where $\mathbb{F}$ is the Fano plane (\cite{Fano}). 
\end{itemize}
More results can see \cite{Balogh}, \cite{Oleg2}, \cite{Nagle}. 

\medskip

In 2020, Balogh, Lemons and Palmer \cite{Balogh2} considered the positive co-degree in hypergraphs, and they
determined the maximum size of an intersecting $r$-graph with minimum positive co-degree
at least $k$ (see also extensions by Spiro \cite{Spiro}).

\begin{defn}
In an $r$-graph $H$, the \emph{minimum positive co-degree} of $H$, denoted $\delta^{+}_{r-1}(H)$, is the minimum co-degree over all $(r-1)$-set $S$ in $H$ such that the co-degree of $S$ is not $0$.
\end{defn}

Very recently, Halfpap, Lemons, and Palmer extended this notion to Tur\'an number in \cite{main}. They defined the positive co-degree Tur\'an number as follows:

\begin{defn}
Let $F$ be an $r$-graph. The \emph{positive co-degree Tur\'an number} $\mathrm{co}^{+}\ex(n,F)$ is the maximum value of $\delta_{r-1}^{+}(H)$ over all $n$-vertex $F$-free $r$-graphs $H$. Moreover, define the \emph{positive co-degree Tur\'an density}

\[\gamma^{+}(F)\eqdef\limsup_{n\rightarrow\infty} \frac{\mathrm{co}^{+}\ex(n,F)}{n}.\]
\end{defn}

In \cite{main}, the authors determined or bounded the positive co-degree Tur\'an density for some small 3-graphs. Let $K_4^{-}$ denote the  $3$-graph $\{abc, bcd, cda\}$, $K_4$ denote the $3$-graph $\{abc, bcd, cda, dab\}$, $C_5^{-}$ denote the $3$-graph $\{abc, bcd, cde, dea\}$, and $C_5$ denote the $3$-graph $\{abc, bcd, cde, dea, eab\}$. For instance, they proved that  \[\gamma^{+}(K_4^{-})=1/3, \qquad \gamma^{+}(F_{3,2})=1/2, \qquad 1/3\le \gamma^{+}(C_5^{-})\le 1/2, \qquad 1/2\le \gamma^{+}(C_5)\le 2/3.  \]  

In this paper, we will determine the positive co-degree number of $C_5$ and $C_5^{-}$ exactly.

\begin{theorem}\label{3}
For $n\ge 6$,
\[\mathrm{co}^{+}\ex(n,C_5^{-})=\lfloor n/3 \rfloor.\]
\end{theorem}

\begin{theorem}\label{4}
For $n\ge 6$,
\[\mathrm{co}^{+}\ex(n,C_5)=\begin{cases} 2k, \quad \text{if $n=4k+0,1,2$;}  \\2k+1, \quad \text{if $n=4k+3$.}  \end{cases}\]
\end{theorem}
\end{section}

\begin{section}{Preliminaries}

Let $H$ be an $3$-graph. For any edge $e=\{x,y,z\}\in E(H)$, we will write $xyz$ for convenience. For any two different vertices $u,v\in H$, denote $N(u,v)$ as the set of vertices $w$ such that $uvw$ forms a 3-edge of $H$, and define $N(u,u)=\varnothing$.
For simplicity, we will say that $a, b, c, d$ form a $K_4$ to represent that $abc, bcd, cda, dab$ form a $K_4$.

For $k\ge 3$, we say $H$ is $k$-partite if there exists a partition $V_1,V_2,\cdots,V_k$ of its vertex
set such that each edge intersects each partition class in at most one vertex. A $k$-partite
hypergraph is \emph{complete} if every possible edge is present and is \emph{balanced} if the class sizes differ
by at most $1$.

We also need the following lemma from \cite[Theorem $4$]{main}.

\begin{lemma}\label{1}
\[\mathrm{co}^{+}\ex(n,K_4^{-})=\lfloor n/3 \rfloor.\]
\end{lemma}

\end{section}

\begin{section}{Proof of Theorem~\ref{3}}
The lower bound is witnessed by the complete balanced $3$-partite $3$-graph. We prove the upper bound then. Let $H$ be an $n$-vertex $3$-graph with $\delta_2^{+}(H)>n/3$. By Lemma~\ref{1}, there are four vertices $v_1, v_2, v_3, v_4$ such that $v_1v_2v_3, v_1v_2v_4, v_1v_3v_4$ form a $K_4^{-}$ in $H$. 

\smallskip
\textbf{Case 1:}  $v_2v_3v_4\notin H$. It follows from $\delta_2^{+}(H)\ge n/3$ that
\[|N(v_2,v_1)\setminus\{v_3,v_4\}|+|N(v_2,v_3)\setminus\{v_1\}|+|N(v_2,v_4)\setminus\{v_1\}|>n-4. \]
So, there exists $v_5\notin \{v_1,v_2,v_3,v_4\}$ which appears in at least two of $N(v_2,v_1), N(v_2,v_3), N(v_2,v_4)$. 
\begin{itemize}
    \item If $v_5\in N(v_2,v_1)\cap N(v_2,v_3)$, then $v_3v_4v_1$, $v_4v_1v_2$, $v_1v_2v_5$, $v_2v_5v_3$ form a $C_5^{-}$;
    \item If $v_5\in N(v_2,v_1)\cap N(v_2,v_4)$, then $v_4v_3v_1$, $v_3v_1v_2$, $v_1v_2v_5$, $v_2v_5v_4$ form a $C_5^{-}$;
    \item If $v_5\in N(v_2,v_3)\cap N(v_2,v_4)$, then $v_4v_1v_3$, $v_1v_3v_2$, $v_3v_2v_5$, $v_2v_5v_4$ form a $C_5^{-}$. 
\end{itemize}
We conclude that $\mathrm{co}^{+}\ex(n,C_5^{-}) = \lfloor n/3 \rfloor$. 

\smallskip
\textbf{Case 2:}  $v_2v_3v_4\in H$. If $\{i,j\}\cup \{k,l\}=\{1,2,3,4\}$, then we write
\[ M_{i,j} \eqdef N(v_i,v_j)\setminus\{v_k,v_l\}.\]
Since $n\ge 6$, it follows from $\delta_2^{+}(H)\ge n/3$ that
\[\sum_{\{i,j\}}|M_{i,j}|\ge 6\left\lceil\frac{n+1}{3} \right\rceil -12> n-4. \]
So, there exists $v_5\notin \{v_1,v_2,v_3,v_4\}$ which appears in at least two of $M_{i,j}$, say $v_5\in M_{i,j} \cap M_{i',j'}$.

\begin{itemize}
    \item If $\{i,j\}\cap \{i',j'\}\neq \varnothing$, then we assume $i=i'=1, \, j=2, \, j'=3$ without loss of generality, and so $v_3v_4v_2$, $v_4v_2v_1$, $v_2v_1v_5$, $v_1v_5v_3$ form a $C_5^{-}$.
    \item If $\{i,j\}\cap \{i',j'\}= \varnothing$, then we assume $i=1, \, j=2, \, i'=3, \, j'=4$ without loss of generality, and so $v_5v_1v_2$, $v_1v_2v_3$, $v_2v_3v_4$, $v_3v_4v_5$ form a $C_5^{-}$.
\end{itemize}
We conclude that $\mathrm{co}^{+}\ex(n,C_5^{-}) = \lfloor n/3 \rfloor$. \qed

\end{section}

\begin{section}{Proof of Theorem~\ref{4}}

The lower bound is witnessed by the complete balanced $4$-partite $3$-graph. We prove the upper bound then. Let $H$ be an $n$-vertex $3$-graph with $\delta_2^{+}(H)\ge n/2$. There are two cases: 
 
\smallskip
 
\textbf{Case 1:} $H$ is $K_4$-free. 

By Lemma~\ref{1}, we can find four vertices $v_1, v_2, v_3, v_4$ such that $v_1v_2v_3, v_1v_2v_4, v_1v_3v_4$ form a $K_4^{-}$.  It follows from $\delta_2^{+}(H)\ge n/2$ that
\[|N(v_1,v_2)|+|N(v_1,v_3)|+|N(v_1,v_4)|+|N(v_2,v_3)|+|N(v_2,v_4)|+|N(v_3,v_4)|\ge 3n. \]
Because $H$ is $K_4$-free, $v_2\notin N(v_3,v_4)$, hence $v_2$ only belongs to $N(v_1,v_3)$ and $N(v_1,v_4)$. Thus, there exists $v_5$ that belongs to at least four of the six sets $N(v_i,v_j) \, (1\le i<j\le 4)$. Here, $v_5\notin \{v_1,v_2,v_3,v_4\}$. Let $m$ be the number of sets among $N(v_1, v_2), \, N(v_1, v_3), \, N(v_1, v_4)$ that contain $v_5$, then $m\ge 1$. 
\begin{itemize}
    \item $m=3$. Since $H$ is $K_4$-free, we have $v_5 \notin N(v_2,v_3) \cup N(v_2,v_4) \cup N(v_3,v_4)$, which contradicts the fact that $v_5$ belongs to at least four of $N(v_i,v_j)$. 
    \item $m=2$. By symmetry, we may assume without loss of generality  that $v_5\in N(v_1,v_2) \cap N(v_1,v_3)$. From the facts that $H$ is $K_4$-free and $v_5\notin N(v_2,v_3)$ we deduce that $v_5\in N(v_2,v_4)\cap N(v_3,v_4)$. We thus obtain a $C_5$ on edges $v_4v_5v_3, v_5v_3v_1, v_3v_1v_2, v_1v_2v_4, v_2v_4v_5$. 
    \item $m=1$. Assume $v_5\in N(v_1,v_2)$ without loss of generality. Then $v_5 \in N(v_2,v_4) \cap N(v_3,v_4)$. It thus follows that $v_1v_2v_5, v_2v_5v_4, v_5v_4v_3, v_4v_3v_1, v_3v_1v_2$ form a $C_5$ in $H$. 
\end{itemize}

\smallskip

\textbf{Case 2:} $v_1, v_2, v_3, v_4$ form a $K_4$ in $H$. 

We first disprove that $\delta_2^{+}(H)>n/2$. Otherwise, since 
\[|N(v_1,v_2)|+|N(v_1,v_3)|+|N(v_1,v_4)|+|N(v_2,v_3)|+|N(v_2,v_4)|+|N(v_3,v_4)|>3n, \]
we can find $v_5$ that belongs to four of the sets $N(v_i,v_j) \, (1 \le i < j \le 4)$. Here, $v_5\notin \{v_1,v_2,v_3,v_4\}$. Group the six sets $N(v_i,v_j)$ into $3$ pairs: 
\begin{equation} \label{eq:pair}
\{N(v_1,v_2),N(v_3,v_4)\}, \quad \{N(v_1,v_3),N(v_2,v_4)\}, \quad \{N(v_1,v_4),N(v_2,v_3)\}.
\end{equation}
Then $v_5$ appears in both sets in at least one of the pairs, say $v_5\in N(v_1,v_2) \cap N(v_3,v_4)$. Assume further that $v_5\in N(v_2,v_3)$, by symmetry. Then $v_1v_2v_5, v_2v_5v_3, v_5v_3v_4, v_3v_4v_1, v_4v_1v_2$ form a $C_5$. 

It follows that $\mathrm{co}^{+}\ex(n,C_5)\le n/2$, and so Theorem~\ref{4} is true if $n \not\equiv 2 \pmod 4$.

However, the main difficulty lies in the case when $n \equiv 2 \pmod 4$. When $n=4k+2$, we have to show that $\mathrm{co}^{+}\ex(n,C_5)=2k$. Indeed, we are going to prove $4 \mid n$ given that $H$ is an $n$-vertex $C_5$-free $3$-graph with $\delta_2^{+}(H)=n/2$. 

Let us state some properties that holds for every four vertices $v_i,v_j,v_k,v_{\ell}$ that form a $K_4$.
It follows that
\[|N(v_i,v_j)|+|N(v_i,v_k)|+|N(v_i,v_{\ell})|+|N(v_j,v_k)|+|N(v_j,v_{\ell})|+|N(v_k,v_{\ell})|\ge 3n. \]
Indeed, the equality has to be achieved, for otherwise we can find a $C_5$ via the analysis above. We thus deduce the following facts: 

\begin{fact}\label{111}
$N(v_i,v_j)=n/2$. Moreover, each $v\in V(H)$ belongs to exactly $3$ of 
\[
N(v_i,v_j), \quad N(v_i, v_k), \quad N(v_i, v_{\ell}), \quad N(v_j, v_k), \quad N(v_j, v_{\ell}), \quad N(v_k, v_{\ell}). 
\]
\end{fact}

If there exists some  vertex $v\in V(H)$ that belongs to two sets in the same pair as in \eqref{eq:pair}, assume that $v\in N(v_i,v_j) \cap N(v_k,v_{\ell})$.  Fact~\ref{111} implies that we can assume further that $v\in N(v_j,v_k)$ by symmetry, and so $v_iv_jv, v_jvv_k, vv_kv_{\ell}, v_kv_{\ell}v_i, v_{\ell}v_iv_j$  form a $C_5$, a contradiction. Hence, from Fact~\ref{111} we see that every $v\in V(H)$ appears in exactly one set in each pair. This implies the following: 

\begin{fact}\label{setminus}
$N(v_i,v_j)=V(H)\setminus N(v_k,v_{\ell})$.
\end{fact}

In order to state some further properties of $v_i, v_j, v_k, v_{\ell}$ that form a $K_4$, we define 
\[
A_i^{jk\ell} \eqdef N(v_j,v_k)\cap N(v_k,v_{\ell})\cap N(v_{\ell},v_j), \qquad B_i^{jk\ell} \eqdef N(v_i,v_j)\cap N(v_i,v_k)\cap N(v_i,v_{\ell}). 
\]
It may seem a bit strange to include a subscript $i$ in $A_i^{jk\ell}$ since the definition of $A_i^{jk\ell}$ has nothing to do with $i$. However, we shall see that $A_i^{jk\ell}$ consists of exactly the vertices $v$ such that $N(v, v_i) = \varnothing$.  

Facts~\ref{111} and \ref{setminus} imply that the $8$ sets $A_i^{jk\ell}, \, A_j^{k\ell i}, \, A_k^{\ell ij}, \, A_{\ell}^{ijk}$ and $B_i^{jk\ell}, \, B_j^{k\ell i}, \, B_k^{\ell ij}, \, B_{\ell}^{ijk}$ partition $V(H)$. Observe that $v_i \in A_i^{jk\ell}$, and so $A_i^{jk\ell} \neq \varnothing$. 

\smallskip

Write $A_i \eqdef A_i^{jkl}$ and $B_i \eqdef B_i^{jkl}$ when the omission of superscripts carries no confusion. 

\begin{fact}\label{AAA}
If $a\in A_i$ and $b\in A_j$, then $N(a,b)=N(v_i,v_j)$. 
\end{fact}

\begin{proof}
Since $a\in A_i$, we obtain a $K_4$ on vertices $a, v_j, v_k, v_{\ell}$. Then Fact~\ref{setminus} implies that \[N(v_i,v_k)=V(H)\setminus(N(v_j,v_{\ell}))=N(a,v_k).\]
Similarly, we have $N(v_i,v_{\ell})=N(a,v_{\ell})$. Since $b\in A_j$, $b\in N(a,v_k)=N(v_i,v_k)$ and $b\in N(v_i,v_{\ell})$, we obtain another $K_4$ on vertices $a, b, v_k, v_{\ell}$. It then follows from Fact~\ref{setminus} that \[N(v_i,v_j)=V(H)\setminus(N(v_k,v_{\ell}))=N(a,b). \qedhere\]
\end{proof}

\begin{fact}\label{kkk}
There exists an integer $q$ depending on $i, j, k, \ell$ such that 
\[
|A_i|-|B_i|=|A_j|-|B_j|=|A_k|-|B_k|=|A_{\ell}|-|B_{\ell}|=q. 
\]
\end{fact}

\begin{proof} By symmetry, it suffices to show that $|A_i|-|B_i|=|A_j|-|B_j|$. From Fact~\ref{111} we deduce
\[|A_i|+|A_k|+|B_j|+|B_{\ell}|=|N(v_k,v_i)|=n/2=|N(v_k,v_j)|=|A_j|+|A_k|+|B_i|+|B_{\ell}|,\]
and so $|A_i|-|B_i|=|A_j|-|B_j|$. 
\end{proof}

\begin{fact}\label{BBB}
If $a\in B_i$, then

\smallskip

(1) $N(a,v_j)\cap (A_{k}\cup B_{k})=\varnothing$, and 

\smallskip

(2) $|A_i|+|A_j|+|B_i|+|B_j|\ge 1+n/2$. 
\end{fact}

\begin{proof}
If $b\in N(a,v_j)\cap A_{k}$, then Fact~\ref{AAA} implies $a\in N(b,v_j)=N(v_k,v_j)$, which contradicts the definition of $B_i$. If $b\in N(a,v_j)\cap B_{k}$, then $av_iv_k, v_iv_kb, v_kbv_j, bv_ja, v_jav_i$ form a $C_5$, a contradiction. We conclude (1). Notice that (1) also implies that $N(a, v_j) \cap (A_{\ell} \cup B_{\ell}) = \varnothing$, by symmetry. To see (2), from (1) we deduce that $\{v_j\}\cup N(a,v_j)\subseteq A_i\cup B_i\cup A_j\cup B_j$. Since $v_i\in N(a,v_j)$, the assumption $\delta_2^+(H) = n/2$ tells us that $|N(a,v_j)| \ge n/2$, and so
\[n/2+1\le |\{v_j\}\cup N(a,v_j)|\le |A_i|+|B_i|+|A_j|+|B_j|. \qedhere\]
\end{proof}

If there exists a $K_4$ on vertices $v_i, v_j, v_k, v_{\ell}$ such that $B_i$, $B_j$, $B_k$, $B_{\ell}$ are all empty, then Fact~\ref{kkk} implies $|A_i| = |A_j| = |A_k| = |A_{\ell}| = q$, $n = 4q$, and so $4\mid n$. If there exists a $K_4$ on vertices $v_i, v_j, v_k, v_{\ell}$ such that $B_i$, $B_j$ are both non-empty, then from Fact~\ref{BBB} we deduce that 
\begin{align*}
|A_i|+|A_k|+|B_i|+|B_k| &\ge n/2+1, \\
|A_j|+|A_l|+|B_j|+|B_l| &\ge n/2+1. 
\end{align*}
By adding the inequalities, we obtain $n \ge n+2$, a contradiction. We henceforth assume that for any $K_4$ in $H$ with vertices $v_i, v_j, v_k, v_{\ell}$, exactly one of $B_i$, $B_j$, $B_k$, $B_{\ell}$ is non-empty. 

\begin{fact}\label{B2}
If $|B_i|=r>0$, then

\smallskip

(1) $n=4q+2r$, 

\smallskip

(2) $|A_i| > n/4 > |A_j|$, 

\smallskip

(3) $N(a,b)=\varnothing$ for every distinct $a, b \in A_j$. 

\end{fact}

\begin{proof}
Since $B_j=B_k=B_{\ell}=\varnothing$, we obtain from Fact~\ref{kkk} that 
\[
|A_i|=q+r, \qquad |A_j|=|A_k|=|A_{\ell}|=q. 
\]
So, $n=|A_i|+|A_j|+|A_k|+|A_{\ell}|+|B_i|+|B_j|+|B_k|+|B_{\ell}|=4q+2r$. This shows (1) and (2). 

Suppose $w\in A_i$. If $w\in N(a,b)$, from Fact~\ref{AAA} we deduce that $a\in N(w,b)=N(v_i,v_j)$ by $w\in A_i$, which contradicts $a\in A_j$. It follows that $N(a,b)\cap A_i=\varnothing$. Similarly, we have $N(a,b)\cap (A_k\cup A_{\ell})=\varnothing$. So, $N(a,b)\subseteq B_i\cup A_j$, and hence
\[N(a,b)\le |A_j|+|B_i|=q+r<n/2.\]
This proves (3), since $\delta_2^+(H) = n/2$. 
\end{proof}

\begin{fact}\label{same}
Suppose $B_i\neq \varnothing$. Then $N(a,v_j)=\varnothing$ if and only if $a\in A_j$. 
\end{fact}

\begin{proof}
According to Fact~\ref{B2}, $a \in A_j$ implies $N(a, v_j) = \varnothing$. If $N(a,v_j)=\varnothing$, then 
\[
a\notin N(v_j,v_i)\cup N(v_j,v_k)\cup N(v_j,v_{\ell}). 
\]
It follows from Fact~\ref{111} that 
\[
a\in N(v_k,v_{\ell})\cap N(v_i,v_{\ell})\cap N(v_i,v_k) = A_j. \qedhere
\]
\end{proof}

\paragraph{Remark.} 
It can be shown that Fact~\ref{same} holds even if we drop the assumption $B_i \neq \varnothing$. However, we omit the proof of this stronger version since we do not need it in the proof of Theorem 4.

\begin{fact}\label{K4}
If $B_i\neq \varnothing$, then for any $a\in B_i$, there exists $b\in B_i$ such that $a, v_i, v_j, b$ form a $K_4$ in $H$. 
\end{fact}

\begin{proof}
Let $r=|B_i|$. Then Fact~\ref{B2} implies $|A_i|=q+r, \, |A_j|=|A_k|=|A_{\ell}|=q, \, n=4q+2r$. 

From Facts~\ref{BBB} and~\ref{B2}, we obtain $N(a,v_j)\cap (A_j\cup A_k\cup A_{\ell})=\varnothing$. So, $N(a,v_j)\subseteq A_i\cup (B_i\setminus\{a\})$, and hence
\begin{equation} \label{eq:q}
|N(a,v_j)\cap (B_i\setminus\{a\})|\ge \frac{n}{2}-|A_i|=q. 
\end{equation}

If $c \in N(a,v_i)\cap A_i$, then Fact~\ref{AAA} implies that $N(c,v_k)=N(v_i,v_k)$, and so $\{a,v_j\}\subseteq N(c,v_k)$. This shows that $v_iac, acv_k, cv_kv_j, v_kv_jv_i, v_jv_ia$ form a $C_5$, a contradiction. So, $N(a,v_i)\cap A_i=\varnothing$. Therefore, we obtain $N(a,v_i)\subseteq A_j\cup A_k\cup A_{\ell}\cup (B_i\setminus\{a\})$, and hence
\begin{equation} \label{eq:r-q}
|N(a,v_i)\cap (B_i\setminus\{a\})|\ge \frac{n}{2}-|A_j|-|A_k|-|A_{\ell}|= r-q. 
\end{equation}

Thus, by writing $X \eqdef B_i \setminus \{a\}$, we deduce from inequalities \eqref{eq:q} and \eqref{eq:r-q} that
\begin{align*}
|N(a,v_i)\cap N(a,v_j)\cap X|&\ge |N(a,v_i)\cap X|+|N(a,v_j)\cap X|-|X|\\
&\ge q+(r-q)-(r-1)=1.
\end{align*}
Pick $b\in N(a,v_i)\cap N(a,v_j)\cap X$. Then $b\in B_i$ and $a, v_i, v_j, b$ form a $K_4$ in $H$. 
\end{proof}

Recall that $v_1, v_2, v_3, v_4$ form a $K_4$ in $H$. Assume without loss of genearlity that $|B_1^{234}|=r_0>0$.

\begin{fact}\label{56}
Suppose $v_5,v_6\in B_1^{234}$ and $v_1, v_2, v_5, v_6$ form a $K_4$. Then $B_1^{256}\neq \varnothing$. 
\end{fact}

\begin{proof}
For any $x\in A_1^{234}$, Fact~\ref{AAA} implies that $v_5\in N(v_1,v_2)=N(x,v_2)$, and hence $x\in N(v_2,v_5)$. Similarly, $x\in N(v_2,v_6)$. Since $x\notin N(v_1,v_2)$, Fact~\ref{setminus} implies $x\in N(v_5,v_6)$, and so $A_1^{234}\subseteq A_1^{256}$. 

If $B_1^{256}=\varnothing$, then from Fact~\ref{B2} we deduce that $n/4<|A_1^{234}|\le|A_1^{256}|<n/4$, a contradiction. 
\end{proof}

Introduce a binary relation $\sim$ on $B_1^{234}$, where $a \sim b$ if and only if $N(a, b) = \varnothing$. 

\begin{fact}\label{class} 
$\sim$ is an equivalence relation. 
\end{fact}

\begin{proof}
Obviously, $\sim$ is both reflective and symmetric. It then suffices to verify the transitivity. 

Assume that $v_5\sim v_7$ and $v_5\sim v_8$. We claim that $N(v_7, v_8) = \varnothing$. By Fact~\ref{K4}, there exists $v_6\in B_1^{234}$ such that $v_1, v_2, v_5, v_6$ form a $K_4$. Since $v_5\sim v_7$ and $N(v_5,v_7)=\varnothing$, Fact~\ref{same} implies $v_7\in A_5^{126}$. Similarly, $v_8\in A_5^{126}$. Note that Fact~\ref{56} implies $B_1^{256}\neq \varnothing$,  hence Fact~\ref{B2} shows $N(v_7,v_8)=\varnothing$. 
\end{proof}

Now we can prove that $4\mid n$, finishing the proof of Theorem~\ref{4}.

Let $q=|A_2^{134}|$. Then Fact~\ref{B2} implies $n=4q+2r_0$, so we only need to show that $2\mid r_0$.

By Fact~\ref{class}, the equivalence relation $\sim$ partitions $B_1^{234}$ into some equivalence classes. For any equivalence class $X$ in $B_1^{234}$ and $v_5\in X$, Fact~\ref{K4} tell us that we can find $v_{6}\in B_1^{234}$ such that $v_1, v_2, v_5, v_{6}$ form a $K_4$. Define $f(X)$ as the equivalence class contained $v_{6}$. By Fact~\ref{56}, $B_1^{256}\neq \varnothing$, hence Fact~\ref{same} implies that $X=A_5^{126}$ and $f(X)=A_6^{125}$. Hence,  for any $v_7\in X$ and $v_8\in B_1^{234}$, $v_1, v_2, v_7, v_8$ forming a $K_4$ is equivalent to $v_8\in f(X)$ by Fact~\ref{AAA}. Thus, the definition of $f$ does not rely on the choice of $v_5$ and $v_6$, and so $f$ is a well-defined function on equivalence classes.

The definition of $f$ implies $f(f(X))=X$ and $f(X)\neq X$ directly. Hence, $f$ induces a $2$-element partition of all equivalence classes. By Fact~\ref{kkk} and Fact~\ref{56}, $|X|=|A_5^{126}|=|A_6^{125}|=f(X)$, so each part of the partition contains an even number of elements, which shows that $2\mid r_0$.
\qed

\end{section}

\paragraph{Acknowledgements}
I am grateful to Oleg Pikhurko for proposing this problem to me, and for many helpful discussions and writing suggestions. I also want to thank Zichao Dong for helpful suggestions on the writing.

\bibliographystyle{plain}
\bibliography{name}

\begin{thebibliography}{10}

\bibitem{Balogh}
J.~Balogh, F.C. Clemen, and B.~Lidickr\'y.
\newblock Hypergraph tur\'an problems in $\ell_2$-norm.
\newblock {\em Preprint, \emph{arXiv:2108.10406.}}, 2021.

\bibitem{Balogh2}
J.~Balogh, N.~Lemons, and C.~Palmer.
\newblock Maximum size intersecting families of bounded minimum positive
  co-degree.
\newblock {\em SIAM Journal on Discrete Mathematics, 35(3)}, 2020.

\bibitem{Oleg}
V.~Falgas-Ravry, E.~Marchant, O.~Pikhurko, and E.~R. Vaughan.
\newblock The codegree threshold for 3-graphs with independent neighborhoods.
\newblock {\em SIAM J. Discrete Math., 29(3):1504–1539}, 2015.

\bibitem{Oleg2}
V.~Falgas-Ravry, O.~Pikhurko, E.~R. Vaughan, and J.~Volec.
\newblock The codegree threshold of {$K_4^{-}$}.
\newblock {\em Preprint, \emph{arXiv:2112.09396}}, 2021.

\bibitem{main}
A.~Halfpap, N.~Lemons, and C.~Palmer.
\newblock Positive co-degree density of hypergraphs.
\newblock {\em Preprint, \emph{arXiv:2207.05639}}, 2022.

\bibitem{Fano}
D.~Mubayi.
\newblock The co-degree density of the fano plane.
\newblock {\em J. Combin. Theory Ser. B, 95(2):333–337}, 2005.

\bibitem{zhao}
D.~Mubayi and Y.~Zhao.
\newblock Co-degree density of hypergraphs.
\newblock {\em J. Combin. Theory Ser. A, 114(6):1118–1132}, 2006.

\bibitem{Nagle}
B.~Nagle.
\newblock Tur\'an related problems for hypergraphs.
\newblock {\em In Proceedings of the Thirtieth Southeastern International
  Conference on Combinatorics, Graph Theory, and Computing, volume 136}, 1999.

\bibitem{Stone}
P.Erd\H{o}s and A.~H. Stone.
\newblock On the structure of linear graphs.
\newblock {\em Bull. Amer. Math. Soc., 52:1087–1091}, 1946.

\bibitem{Spiro}
S.~Spiro.
\newblock On $t$-intersecting hypergraphs with minimum positive codegrees.
\newblock {\em Preprint, \emph{arXiv:2110.10406.}}, 2021.

\end{thebibliography}

\end{document}